\documentclass[10]{amsart}
\usepackage{amssymb}
\usepackage{amsmath}
\usepackage{hyperref}
\usepackage{color}
\newcommand{\bt}{\begin{theorem}}
\newcommand{\et}{\end{theorem}}
\newcommand{\bi}{\begin{itemize}}
\newcommand{\ei}{\end{itemize}}
\newcommand{\bea}{\begin{eqnarray}}
\newcommand{\ba}{\begin{array}}
\newcommand{\eea}{\end{eqnarray}}
\newcommand{\ea}{\end{array}}
\newcommand{\be}{\begin{equation}}
\newcommand{\ee}{\end{equation}}

\newcommand{\lgra}{\longrightarrow}%
\newcommand{\what}{\widehat}%
\newcommand{\wtilde}{\widetilde}%
\newcommand{\Ab}{\mathcal A}%
\newcommand{\R}{\mathbb R}%
\newcommand{\C}{\mathbb C}%
\newcommand{\Z}{\mathbb Z}%
\newcommand{\N}{\mathbb N}%

\newtheorem{theorem}{Theorem}[section]

\theoremstyle{definition}

\newtheorem{conjecture}[theorem]{Conjecture}

\theoremstyle{definition}

\numberwithin{equation}{subsection}
\numberwithin{theorem}{subsection}

\makeindex
\begin{document}
\baselineskip16pt
\author[S. K. Ray]{Swagato K. Ray }
\address[S. K. Ray]{Stat-Math Unit, Indian Statistical
Institute, 203 B. T. Rd., Calcutta 700108, India}
\email{swagato@isical.ac.in}

\author[R. P. Sarkar]{Rudra P. Sarkar}
\address[R. P. Sarkar]{Stat-Math Unit, Indian Statistical
Institute, 203 B. T. Rd., Calcutta 700108, India}
\email{rudra@isical.ac.in}


\title[Roe's  Theorem]{A theorem of  Roe and Strichartz for Riemannian symmetric spaces of noncompact type}
\begin{abstract} Generalizing a result of Roe  \cite{Roe} Strichartz proved in \cite{Str}  that
if  a doubly-infinite sequence $\{f_k\}$ of functions on $\R^n$ satisfies $f_{k+1}=\Delta f_k$ and  $|f_{k}(x)|\leq M$ for all $k=0,\pm 1,\pm 2,\cdots$ and $x\in \R^n$,
then $\Delta f_0(x)= -f_0$.  Strichartz also showed that the result  fails for hyperbolic $3$-space. This negative result can be
indeed extended to any Riemannian symmetric space of noncompact type.
Taking this into account we shall prove that for
all  Riemannian symmetric spaces of noncompact type the theorem actually holds
true when uniform boundedness is modified suitably.
 \end{abstract}
 \maketitle
\section{Introduction}
Generalizing a result of Roe \cite{Roe},   Strichartz  (\cite{Str}) proved the following theorem on $\R^n$. (See also \cite{Howd-Reese} and the references therein.)
\begin{theorem}[Strichartz] \label{str-thm}
Let $\{f_j\}_{j\in\Z}$ be a doubly infinite sequence of measurable functions on
$\R^n$ such that for all
$j\in\Z$, $(\mathrm{i})$ $\|f_j\|_{L^{\infty}(\R^n)}\leq C$ for some constant $C>0$  and $(\mathrm{ii})$ for some $\alpha>0$,  $L f_j=\alpha f_{j+1}$ where  $L=\sum_{i=1}^n\frac{\partial^2}{\partial x_i^2}$. Then $L f_0=-\alpha f_0.$\label{bob}
\end{theorem}
 Strichartz also observed in the same paper \cite{Str} that  the result holds true for Heisenberg groups $\mathbb H^n$, but  fails for hyperbolic $3$-space. A slight generalization of the counter example given in \cite{Str} shows that in any Riemannian symmetric space $X$ of noncompact type there is a sequence of functions $\{f_j\}$ which satisfies the hypothesis (where the Laplace-Beltrami operator $\Delta$ on $X$ replaces  $L$), but $f_0$ is not an eigenfunction of $\Delta$  (see \cite{KRS-2}). We take this negative result as our starting point. Aim of this paper is to prove an analogue of Strichartz's result for all Riemannian symmetric spaces $X=G/K$ of noncompact type.

From the counter example provided by Strichartz in \cite{Str} it is not difficult to perceive  that the failure is influenced by the spectral properties of the Laplacian $\Delta$ of the symmetric space. More precisely the failure is due to the difference  between the $L^2$ and $L^\infty$-spectrum  of $\Delta$, which in turn depends on the exponential volume growth of the underlying manifold $X$. This sets the task of searching for a possible analogue  conducive to the structure of the space. We began our study in this direction with the rank one symmetric space in \cite{KRS-2}, where  the situation was saved, substituting  the $L^\infty$-norm by the weak $L^2$-norm. However the use of weak $L^2$-norm seems to be restrictive to the rank one case. The following observation can be considered as  a first indication of this. We recall that the foremost examples of eigenfunctions of $\Delta$ are  the elementary spherical functions $\phi_\lambda$ with $\lambda\in \mathfrak a^\ast$. Unlike the rank one case, in general rank,  these eigenfunctions    do not belong to the weak $L^2$-space.   We also recall that on $X$ the objects  which correspond to $e^{i\langle \lambda, x\rangle}$ (on $\R^n$) are $e_{\lambda, k}: x\mapsto e^{-(i\lambda+\rho)H(x^{-1}k)}, k\in K, \lambda\in \mathfrak a^\ast$. They are the basic eigenfunctions of $\Delta$  which are constant on each horocycles $\xi_{k, C}=\{x\in G/K\mid H(x^{-1}k)=C\}$; but  unlike their counterparts in the Euclidean case, $e_{\lambda, k}$ are not $L^\infty$-functions, in particular they do not satisfy the hypothesis of Strichartz's theorem. As in the Euclidean case,   we have another set of prominent eigenfunctions namely  the Poisson transforms $P_\lambda F(x)=\int_{K/M} e_{\lambda, k}(x)F(k)dk, \lambda\in \mathfrak a^\ast$ of $L^p$-functions $F$ on the boundary $K/M$ of the space $X$  for some $p\ge 1$.
Taking all these into account we motivate ourselves  to look for a size-estimate which accommodates a large class of eigenfunctions of $\Delta$, including those mentioned above. One such is the so-called ``Hardy-type'' norm introduced in \cite{Stoll}  and used  effectively in \cite{BOS}. We shall use this norm to formulate  our main result, which is stated below.
(For any unexplained notation see Section 2.)

\begin{theorem} \label{symm-main}
Let $\{f_j\}_{j\in\Z}$ be a doubly infinite sequence of measurable functions on
$X$ such that for some  real number $\alpha$ and for all
$j\in\Z$:
 \begin{enumerate}
 \item[(i)] $\Delta f_j=(\alpha^2+|\rho|^2) f_{j+1}$,
 \item[(ii)] for   a fixed $p\ge 1$, $\|f_j(\cdot a)\|_{L^p(K)}\le C_{p}  \phi_0(a)$ for all $a\in \overline{A^+}$ and for a  constant $C_{p}>0$ depending only on  $p$.
 \end{enumerate}
Then $\Delta f_0=-(\alpha^2+|\rho|^2) f_0$. In particular if $\alpha=0, p>1$ then $f_0(x)=P_0F(x)=\int_K e^{-\rho(H(x^{-1}k)}F(k)dk$ for some $F\in L^p(K/M)$ and if $\alpha=0, p=1$ then $f_0=P_0\mu(x)=\int_K e^{-\rho(H(x^{-1}k)}d\mu(k)$ for some signed measure $\mu$ on $K/M$.
\end{theorem}

The particular case of  $p=\infty$ in the condition (ii) of the  hypothesis  simplifies as $|f_j(x)|\le C\phi_0(x)$ for all $x\in X$ and $j\in \Z$, and is close to the estimate used in Theorem \ref{bob}.
We take a $\lambda\in \mathfrak a^\ast$ with  $|\lambda|^2=\alpha^2$. Then for any fixed $k\in K$, the sequence of functions $\{(-1)^j\,e_{\lambda, k}\}_{j\in \Z}$   satisfy the hypothesis with   $p=1$.  For such a $\lambda$ and for any $1\le p<\infty$, the sequence of Poisson transforms $\{(-1)^j\, P_\lambda F\}_{j\in \Z}$ for any $F\in L^p(K/M)$ also satisfy the hypothesis. (See Theorem \ref{Ben} in Section 2.)

The result above may  also be viewed from the following perspective. In \cite{BOS} Ben sa\"{i}d et. al. used the Hardy-type norm to characterize a large family of eigenfunctions  of $\Delta$  as Poisson integral of functions in $L^p(K/M)$. However their result does not apply to the case where the eigenvalues are of the form $|\lambda|^2+|\rho|^2, \lambda\in \mathfrak a^\ast, \lambda\neq 0$. This makes the situation very close to the Euclidean in the following sense. While it is well known that bounded harmonic functions on $\R^n$ are constants,  there is no simple characterization of bounded eigenfunctions of $L$ on $\R^n$ with nonzero real eigenvalues.  Strichartz's result on the other hand deals with the latter case. Analogously, Theorem \ref{symm-main} endeavors to ``capture''  eigenfunctions of $\Delta$ with eigenvalues of the form mentioned above.

The space $X=G/K$ enjoys a dichotomy as it can be viewed as a  solvable  Lie group,  $S=N\rtimes A$ through the Iwasawa decomposition of $G=NAK$. The group $S$ is  amenable  like $\R^n$ and $\mathbb H^n$, though unlike them $S$ is nonunimodular. On $S$, one considers a second order right $S$-invariant differential operator $\mathcal L$ which is known as the {\em distinguished Laplacian}. Unlike $\Delta$, the $L^p$-spectrum of $\mathcal L$ for $1\le p<\infty$ coincides with the $L^2$-spectrum (see e.g. \cite{GM}).  This intrigues us to formulate a version of Theorem \ref{symm-main} substituting $\Delta$ by  $\mathcal L$, which is the next result.
\begin{theorem}\label{thm-distinguished}
Let  $\{f_j\}_{j\in \Z}$ be a doubly infinite sequence of measurable
functions defined on $S$.  If for some  real number $\alpha>0$ and for some  $C>0$,
\[{\mathcal L} f_j=
(\alpha^2+|\rho|^2) f_{j+1},\quad |f_j(x)|\le C  \delta(x), \text{ for all } x\in S, \text{ for all } j\in \Z,\]  where $\delta$ is the modular function of $S$, then
$\mathcal L f_0= \alpha f_0$.
\end{theorem}
Both of these results  may be viewed as  ``exact'' analogues of the Euclidean theorem. For Theorem \ref{symm-main}, one can argue that any reasonable analogue of the object $\phi_0$ on $\R^n$ is the constant function $1$,   while for Theorem \ref{thm-distinguished} one may recall that  for $\R^n$ (and  $\mathbb H^n$), $\delta\equiv 1$.

\section{Notation and Preliminaries}  For two positive functions $f_1$  and $f_2$  we shall write $f_1\asymp f_2$ to mean there are positive constants $C_1, C_2$ such that $C_1f_1\le f_2\le C_2 f_1$.  For a measurable function $f$ on $\R^n$ we define its Euclidean Fourier transform at $\lambda\in \C^n$ by
\[\wtilde{f}(\lambda)=\int_{\R^n} f(x)e^{-i\lambda\cdot x} dx,\] (where $\lambda\cdot x$ is the Euclidean (real) inner product of $\lambda$ and $x$), whenever the integral converges. Let $S(\R^n)$ be the Schwartz space on $\R^n$. Precisely $S(\R^n)$ is the set of functions in $C^\infty(\R^n)$ such that $\mu_{r,s}(f)<\infty$ for all multiindex $r$ and $s>0$ where
 \[\mu_{r,s}(f)=\sup_{x\in \R^n} (1+|x|)^s |D_rf(x)|.\] Here  $D_r=\frac{\partial^{r_1}}{\partial x_1}\ldots \frac{\partial^{r_n}}{\partial x_n}, r=(r_1,\dots,r_n)$ is a differential operator and $|x|$ is the Euclidean norm of $x$. The space $S(\R^n)$ becomes a  Frechet space with respect to the topology generated by the seminorms $\mu_{r,s}$.  The  dual space of $S(\R^n)$ is called the space of tempered distributions which will be denoted by $S(\R^n)'$. The following facts are well known: (1) $f\mapsto \wtilde{f}$ is an isomorphism from $S(\R^n)$ to itself, (2) using this isomorphism one can extend the notion of Fourier transform and derivative to $S(\R^n)'$, (3) $L^p$-functions for $1\le p\le \infty$ are tempered distributions.

The required  preliminaries and notation related to the noncompact semisimple Lie groups and the
associated symmetric spaces are standard and can be found for example
in \cite{GV, Helga-2, Helga-3}. To make the article self-contained  we shall gather only those  results which will be used throughout  this paper.
Let $G$ be a noncompact connected semisimple Lie group with finite centre and $K$ be a maximal compact subgroup of $G$. Let $X=G/K$ be the associated Riemannian symmetric space of noncompact type.  We let $G=KAN$ denote a fixed
Iwasawa decomposition of $G$. Let $\mathfrak g$, $\mathfrak k$,
$\mathfrak a$ and $\mathfrak n$ denote the Lie algebras of $G$,
$K$, $A$ and $N$ respectively. Let $\mathfrak g_\C$ be the complexificaion of $\mathfrak g$ and $\mathcal U(\mathfrak g_\C)$ be its universal enveloping algebra.
We recall that the elements of $\mathcal U(\mathfrak g_\C)$ are identified with the left-invariant differential operators on $G$ and there exists an anti-isomorphism $\imath$ from $\mathcal U(\mathfrak g_\C)$ to algebra of right-invariant  differential operators on $G$.
  We choose and keep fixed
throughout a system of positive restricted roots for the pair
$(\mathfrak g, \mathfrak a)$, which we denote by $\Sigma^+$.
The multiplicity of a root $\alpha\in \Sigma^+$ will
be denoted by $m_\alpha$. As usual the half-sum of the elements of
$\Sigma^+$ counted with their multiplicities will be denoted by
$\rho$. Let $H:G\lgra \mathfrak a$ be the Iwasawa projection
associated to the Iwasawa decomposition, $G= KAN$. Then $H$ is
left $K$-invariant and right $MN$-invariant where $M$ is the
centralizer of $A$ in $K$.  The Weyl group of the pair $(G, A)$
will be denoted by $W$. Let $\mathfrak a^*$ be the real dual of
$\mathfrak a$ and $\mathfrak a^*_\C$ its complexification. Let $\mathfrak a_+$ (respectively $\mathfrak a_+^\ast$) denote
the positive Weyl chamber in $\mathfrak a$ (respectively $\mathfrak a^\ast$).
Let $A^+=\exp {\mathfrak a}_+$  and $\overline{A^+}$ be the closure of $A^+$.

We recall that the Killing form $B$ restricted to $\mathfrak a$ is
a positive definite inner product on $\mathfrak a$ and it gives a
$W$-equivariant isomorphism of $\mathfrak a$ with $\mathfrak a^*$.
For $\lambda\in \mathfrak a^*$ we denote the corresponding element
in $\mathfrak a$ by $H_\lambda$. Let $\dim \mathfrak a=l$, i.e. $l$ is  the rank of $X$. We
will identify $\mathfrak a$ and $\mathfrak a^*$  with $\R^l$, as
an inner product space, with the inner product on $\R^l$ being the
pull-back of the Killing form. This inner product on $\R^l$ as
well as on $\mathfrak a$, $\mathfrak a^*$ will be referred to as
the Killing inner product and will be denoted by $\langle\, ,\,
\rangle$. The associated norm will be denoted by $|\cdot|$. We
hope that this symbol will not be confused with the absolute value
symbol. Since  $\exp:\mathfrak a\lgra A$ is an isomorphism, as a
group $A$ can be identified with $\R^l$.

For $x\in G$, we define $\sigma(x)=d(xK, K)$ where $d$ is the
canonical distance function for $X=G/K$ coming from the Riemann
metric on $X$ induced by the Killing  form restricted to
$\mathfrak p$. Here $\mathfrak g=\mathfrak k\oplus\mathfrak p$
(Cartan decomposition) and $\mathfrak p$ can be identified with
the tangent space at $eK$ of $G/K$. The function $\sigma(x)$ is
$K$-biinvariant and continuous. Note that for $x=k_1ak_2$ (polar decomposition),
$k_1,k_2\in K$, $a\in \overline{A^+}$, $\sigma(x)=\sigma(a)=|\log
a|$, the Killing norm of $\log a$, where $\log a$ is the unique
element in ${\mathfrak a}$ such that $\exp(\log a )=a$.

On $X$ we fix the measure $dx$  which is induced by the metric we
obtain from $B$. As the metric is $G$-invariant, so is $dx$.  On
$G$ we fix the Haar measure $dg$ satisfying
$$\int_Xf(x)dx=\int_gf(g)dg,$$ for every integrable function $f$
on $X$ which we also consider as a right $K$-invariant function on
$G$. While dealing with functions on $X$, we may slur over the
difference between the two measures.

Through the identification of $A$ with $\R^l$ we use the Lebesgue
measure on $\R^l$ as the  Haar measure $da$ on $A$. As usual on
the compact group $K$ we fix the normalized Haar measure $dk$
(i.e. vol$(K)=\int_K dk=1$). Finally we fix the Haar measure $dn$
on $N$ by the condition that
$$\int_G f(g)dg=\int_A\int_N\int_Kf(ank)\,dk\,dn\,da$$
holds for every integrable function $f$ on $G$.

Following integral formulae correspond to the Iwasawa and polar  decompositions respectively. For  any $f\in C_c^\infty(G)$,
\[\int_G f(x)dx=\int_K\int_{\mathfrak a}\int_N f(k\exp H n)e^{2\rho(H)}dn\, dH\, dk \text{ and}\]
\[\int_G f(x)dx=\int_K\int_K\int_{\mathfrak a^+} f(k_1\exp H k_2)J(H) dH\, dk_1\, dk_2,\]  where $dH$ is the Lebesgue measure of $\R^l$ with which $\mathfrak a$ is identified, $dn$, $dk$ are the normalized Haar measures of $K$,  $N$ respectively and
\[J(H)= C\prod_{\alpha\in \Sigma^+}(\sinh \alpha(H))^{m_\alpha}, H\in \mathfrak a^+,\]  $C$ being a normalizing constant.

For a measurable function $f$ on $X$ we define its $K$-invariant part $\mathcal K(f)(x)=\int_Kf(kx)dk$. We shall call $\mathcal K$ the $K$-averaging  operator. A function on $X$ is $K$-invariant if $f(kx)=f(x)$, for all $x\in X, k\in K$, equivalently $f(x)=\mathcal K(f)(x)$. We note that $\int_X \mathcal K(f)(x)g(x)dx =\int_X f(x) \mathcal K(g)(x) dx$ whenever the integrals on both sides converge. It follows that if $\mathcal K(f)=0$ and $\mathcal K (g)=g$ then $\int_Xf(x)g(x)dx=0$.

For $\lambda\in \mathfrak a^\ast_\C=\C^l$, the elementary spherical function $\phi_\lambda$ is given by \[\phi_\lambda(x)=\int_Ke^{-(i\lambda+\rho)H(x^{-1}k)}dk.\] It is a $K$-biinvariant eigenfunction of the Laplace-Beltrami operator  $\Delta$;   $\Delta\phi_\lambda=-(|\lambda|^2+|\rho|^2)\phi_\lambda$ and $\phi_\lambda=\phi_{w\lambda}$ for $w\in W$. In fact, $\phi_\lambda$ are particular examples of more general class of eigenfunctions called the Poisson transforms  $P_\lambda F$ defined in the introduction. Our main results will use certain size-estimates and characterizations of these Poisson transforms.
 In this regard, we quote the following  particular cases of a more general result proved in \cite[Proposition 3.4, Proposition 3.6]{BOS}.
\begin{theorem}[Ben Said et. al.]\label{Ben}
Let  $1\le p\le \infty$ be fixed  and $F\in L^p(K/M)$. Then
\[\sup_{x\in X}\phi_0(x)^{-1}\left[\int_K|P_\lambda F(kx)|^p dk\right]^{1/p}=\|F\|_{L^p(K/M)} \text{ when } 1\le p<\infty\text{ and }\]
\[\sup_{x\in X}\phi_0(x)^{-1}|P_\lambda F(x)| dk=\|F\|_{L^\infty(K/M)} \text{ when } p=\infty.\]
Moreover, suppose that  a function $f$ on $G/K$ satisfies $\Delta f=-|\rho|^2 f$ and $\sup_{x\in X}\phi_0(x)^{-1}\left[\int_K|f(kx)|^p dk\right]^{1/p}<\infty$ for some $1\le p< \infty$ or $\sup_{x\in X}\phi_0(x)^{-1}|f(x)|<\infty$, then there exists unique $F\in L^p(K/M)$ when $1<p\le \infty$ and a signed measure $\mu$ when $p=1$ such that $f=P_0F$ or $f=P_0\mu$.
\end{theorem}

For a suitable function $f$ on $G/K$ its spherical Fourier transform $\what{f}(\lambda)$ for $\lambda\in \C^l$ is defined by \[\what{f}(\lambda)=\int_Gf(x)\phi_\lambda(x)dx.\] It is then clear that $\what{f}(\lambda)=\what{f}(w\lambda)$ for all $w\in W$.

We now need to introduce the notion of Schwartz spaces and distributions on $X$.
The $L^2$-Schwartz space $C^2(G)$ is defined as the set of all $C^\infty$-functions on $G$ such that
$$\gamma_{r,g_1, g_2}(f)=\sup_{x\in G}|(\imath g_1)g_2\,f(x)|
\phi_0(x)^{-1}(1+\sigma(x))^r<\infty,$$ for all nonnegative integers $r$
and $g_1, g_2\in \mathcal U(\mathfrak g_\C)$. Let
$C^2(G//K)$ be the set of $K$-biinvariant  functions in $C^2(G)$. We recall that (\cite{Ank1}) the spherical Fourier transform is an isomorphism  from $C^2(G//K)$ to $S(\R^l)_W$, where $S(\R^l)_W$ is the subspace of W-invariant functions in $S(\R^l)$.

The dual space of $C^2(G/K)$ will be denoted by $C^2(G/K)'$ and its elements will be called $L^2$-tempered distributions.  The translation of $T\in C^2(G/K)'$ by an element $y\in G$ and its convolution with a function $g\in C^2(G//K)$ are  defined in the usual way by $(\ell_y T)f=T(\ell_{y^-1} f)$ and  $T\ast g(y)=(\ell_yT)(g)$, where $\ell_yf(x)=f(y^{-1}x)$. An $L^2$-tempered distribution $T\in C^2(G/K)'$ is called $K$-invariant if
$\langle T, \psi\rangle=\langle T, \mathcal K(\psi)\rangle$, $\psi\in C^2(G/K)$. The set of $K$-invariant $L^2$-tempered distributions on $G/K$ will be denoted by $C^2(G//K)'$.
The heat kernel $h_t$ for $t>0$ is a $K$-invariant function in $C^2(G/K)$ which is defined using the isomorphism of $C^2(G//K)$ with $S(\R^l)_W$, prescribing its spherical Fourier transform   $\what{h_t}(\lambda)=e^{-t(|\lambda|^2+|\rho|^2)}$. It is well known that $h_t\in C^2(G//K)$. Using \cite[Theorem 4.1.1]{egu79} it can be verified that  $f\ast h_t\to f$ in $C^2(G/K)$ as $t\to 0$ and hence $T\ast h_t\to T$ as $t\to 0$ in the sense of distribution.

Finally we need the notion of Abel transform.
For a  $K$-invariant function $f$ on $G/K$ its Abel transform $\Ab f$ is defined by:
\[\Ab f(a)=e^{\rho(\log a)} \int_N f(an)dn,\text { for } a\in A,\] whenever the integral makes sense.   We recall that the {\em slice projection} theorem (\cite{Ank1}) states that for any $f\in C^2(G//K)$, $\lambda\in \mathfrak a^\ast\equiv \R^l$, we have the identity \[\wtilde{\Ab f}(\lambda)=\what{f}(\lambda).\]
The following theorem proved in \cite{Ank1} will be crucial for this paper.
\begin{theorem}
The Abel transform $\Ab: C^2(G//K)\to S(\R^l)_W$ is a topological isomorphism
\end{theorem}
The use of Abel transform in our proof is somewhat similar to that of \cite{Bag-Sita} (see also \cite{Helga-Abel}).

\section{Roe-Strichartz theorem for  Laplace-Beltrami operator}
\subsection{Distribution-version of the Euclidean theorem}
  Since   the real rank of $G$ is $l$, it follows that $W$ is a subgroup of $O(l)$ as $W$ preserves the inner product induced by the Killing form.  For $T\in S(\R^l))'$ and $w\in W$ we define
  \[(wT)f= T(wf)\text{ for all } f\in S(\R^l), \text{ where } wf(x)=f(wx).\]
  The $W$-invariant component $f^\#$ of $f\in S(\R^l)$ (respectively $T^\#$ of   $T\in S(\R^l)'$) is defined  as
  \[f^\#=\frac1{|W|}\sum_{w\in W} wf, (\text{respectively } T^\#=\frac1{|W|}\sum_{w\in W} wT) \text{ where } |W| \text{ denotes the cardinality of } W. \]
  A tempered distribution $T$ is called $W$-invariant if $T^\#=T$. It is easy to verify that for $T\in S(\R^l)'$ and $f\in S(\R^l)$, $T^\# f=Tf^\#$ and in particular when  $T$ is  $W$-invariant $Tf=Tf^\#$. It is also not difficult to see that the Laplacian $L$ of $\R^l$ commutes with $W$-action  and hence  if $f\in S(\R^l)_W$ (respectively $T\in S(\R^l)_W'$)
then $L f\in S(\R^l)_W$ (respectively $L T\in S(\R^l)_W'$).

We shall first prove the following version of the Euclidean theorem. Below $L_1=L-|\rho|^2$.
\begin{theorem}\label{Euclid-dist} Let
$\{T_j\}$ be  a doubly infinite sequence of $W$-invariant tempered
distributions  on $\R^l$ such that

$(i)$ $L_1 T_j=zT_{j+1}$ for some $z\in \C$, $|z|\ge |\rho|^2$.

$(ii)$ for all $\psi\in S(\R^l)_{W}$, $|\langle T_j, \psi\rangle|\le
M\mu(\psi)$ for some fixed seminorm $\mu$ of $S(\R^l)$ and $M>0$.

Then $L_1 T_0=-|z| T_0$.
\end{theorem}
This theorem  is essentially proved in \cite{Str, Howd-Reese}.
For the sake of completeness, we include here only a very brief sketch of the argument.
\begin{proof}
Since $|z|\ge |\rho|^2$, we write $z=(\alpha^2+|\rho|^2)e^{i\theta}, \alpha\ge 0, \theta\in \R$. For a $T\in S(\R^l)'$, by $\wtilde{T}$ we denote its Euclidean Fourier transform.
For $T_0$ to be an eigendistribution of $L_1$ with eigenvalue $-|z|$, it is necessary that  the distribution $\wtilde{T_0}$ is supported on the sphere $\{x\in \R^l \mid |x|=\alpha^2+|\rho|^2\}$. First we shall  prove that.  Then following exactly the steps of \cite[p.210]{Howd-Reese} one  can show that there exists $N\ge 0$ such that $(L_1+(\alpha^2+|\rho|^2))^{N+1}\wtilde{T_0}=0$, which will finally lead to $(L_1+(\alpha^2+|\rho|^2))\wtilde{T_0}=0$ (see \cite[p. 210--211]{Howd-Reese} for details).

   It follows from the  hypothesis that
\[\wtilde{T_0}=(-1)^j\, e^{ij\theta}\left(\frac{\alpha^2+|\rho|^2}{|x|^2+|\rho|^2}\right)^j\wtilde{T_j},\] where $x$ is a dummy variable.
We take a function  $\phi\in S(\R^l)$ such that  $\mathrm{Support}\,\phi\subseteq\{x\in \R^l \mid |x|\ge \alpha+\varepsilon\}$, for some $\varepsilon>0$. Hence $\mathrm{Support}\, \phi^\#\subseteq \{x\in \R^l \mid |x|\ge \alpha+\varepsilon\}$. Let $\psi$ be the Euclidean Fourier transform of the function $(\alpha^2+|\rho|^2)^j(|x|^2+|\rho|^2)^{-j}\phi$. Then
\begin{equation}|\langle\wtilde{T_0}, \phi\rangle|=|\langle T_j, \psi\rangle|=|\langle T_j,\psi^\#\rangle| \le M\mu(\psi^\#)\le M\gamma_{\beta, \tau}\left[\left(\frac{\alpha^2+|\rho|^2}{|x|^2+|\rho|^2}\right)^j\phi^\#\right]
\label{euclidean-2}
\end{equation}
where $\gamma_{\beta, \tau}(f)=\sup_{x\in \R^l}(1+|x|)^\beta|D^\tau f(x)|$ for some positive integer $\beta$ and multi index $\tau$. It follows from the fact that $|x|\ge \alpha+\varepsilon$ on the support of $\phi$, that the right hand side of \eqref{euclidean-2} goes to zero as $j\to \infty$.
 A similar argument taking $j\to -\infty$ will show that $\langle\wtilde{T_0}, \phi\rangle=0$ for all $\phi\in S(\R^l)$ with  support of $\phi\subseteq\{x\in \R^l \mid |x|\le \alpha-\varepsilon\}$.  This proves that distributional support of $\wtilde{T_0}$ is contained in the sphere $\{x\in \R^l \mid |x|=\alpha\}$.
\end{proof}
\subsection{Distribution-version of the  theorem for the  symmetric spaces}
First we shall prove a version of the Roe-Strichratz theorem for $K$-invariant tempered distributions and then we shall generalize the result for arbitrary tempered distributions.
\begin{theorem} \label{Symm-dist-radial}
 If for a doubly infinite sequence
$\{T_j\}$ of $K$-invariant $L^2$-tempered distributions on $X$, $\Delta T_j=
zT_{j+1}$ for some $z\in \C$ with $|z|\ge |\rho|^2$ and for a fixed
seminorm $\nu$ of $C^2(X)$, $|\langle T_j, \phi\rangle|\le
M\nu(\phi)$ for some $M>0$ for all $\phi\in C^2(G//K)$, then $\Delta
T_0=-|z| T_0$.
\end{theorem}
\begin{proof}

 Since  $\Ab: C^2(G//K)\rightarrow
S(\R^l)_{W}$ is an isomorphism, its adjoint  $\Ab^\ast:
S(\R^l)'_{W} \to C^2(G//K)'$ and $B=(\Ab^\ast)^{-1}:
C^2(G//K)'\rightarrow S(\R^l)'_{W}$ are isomorphisms (see \cite[p. 541]{Helga-abel}).

We claim that for $T\in C^2(G//K)'$, $B(\Delta T)=L_1BT$. We  note that    for any $g\in C^2(G//K)$, $L_1\Ab g=\Ab \Delta g$.
Indeed  by the slice-projection theorem (see section 2) $\wtilde{\Ab g}(\lambda)=\what{g}(\lambda)$  for any $\lambda\in \mathfrak a^\ast$.  Therefore the Euclidean Fourier transform of $L_1 \Ab g$ is $-(|\lambda|^2+|\rho|^2)\wtilde{\Ab g}(\lambda)=-(|\lambda|^2+|\rho|^2)\what{g}(\lambda)$. Again by slice-projection theorem, Euclidean Fourier transform of $\Ab \Delta g$ at $\lambda$ is $\what{\Delta g}(\lambda)=-(|\lambda|^2+|\rho|^2)\what{g}(\lambda)$. The assertion now follows from the injectivity of the Fourier transform.
Using this we get for $g\in C^2(G//K)$ and $\Ab g=F\in S(\R^l)_W$,
\[\langle L_1BT, F\rangle =\langle L_1BT, \Ab g\rangle
=\langle BT, L_1\Ab g\rangle =\langle BT, \Ab \Delta g\rangle=\langle \Ab^\ast BT, \Delta g\rangle=\langle
T,  \Delta g\rangle =\langle \Delta T,   g\rangle\]
\[=\langle \Ab^\ast B \Delta T,   g\rangle
=\langle B \Delta T,   \Ab g\rangle =\langle B \Delta T,
F\rangle.\] This shows that $L_1 BT=B(\Delta T)$

The condition  $\Delta T_j= z T_{j+1}$ implies $B(\Delta T_j)=z
B(T_{j+1})$. Applying the identity $B(\Delta T)=L_1BT$ we have $L_1
BT_j=z BT_{j+1}.$

Next we shall show that there exists a fixed seminorm $\mu$ of
$S(\R^l)_W$ such that $|\langle BT_j, \psi\rangle|\le M\mu(\psi)$
for all $\psi\in S(\R^l)_W$. Indeed, using that $\Ab:C^2(G//K)\to
S(\R^l)_W$ is an isomorphism, for every $\psi\in S(\R^l)_W$ we have
a $\phi\in C^2(G//K)$ such that $\Ab(\phi)=\psi$ and a seminorm
$\mu$ on $S(\R^l)_W$ such that $\nu(\phi)\le \mu(\psi)$ for all such
pairs $\phi\in C^2(G//K)$ and $\psi\in S(\R^l)_W$. Hence,
\[ |\langle BT_j, \psi\rangle|=|\langle BT_j, \Ab
\phi\rangle|=|\langle T_j, \phi\rangle|\le M\nu(\phi)\le M
\mu(\psi).\] Thus the sequence $\{BT_j\}$ satisfies the hypothesis
of  Theorem \ref{Euclid-dist} and hence
\[L_1 BT_0=-|z|
BT_0.\] Using again the identity $B(\Delta T)=L_1BT$ we get $B(\Delta
T_0)=B(-|z| T_0)$. Since $B$ is injective we have, $\Delta
T_0=-|z|T_0$. This completes the proof.
\end{proof}
Now we shall withdraw the condition that $T_j$ are $K$-invariant.
\begin{theorem} \label{Symm-dist}
 If for a doubly infinite sequence
$\{T_j\}$ of $L^2$-tempered distributions on $X$, $\Delta T_j=
zT_{j+1}$ for some $z\in \C$ with $|z|\ge |\rho|^2$ and for a fixed
seminorm $\nu$ of $C^2(X)$, $|\langle T_j, \phi\rangle|\le
M\nu(\phi)$ for some $M>0$ for all $\phi\in C^2(X)$, then $\Delta
T_0=-|z| T_0$.
\end{theorem}
\begin{proof}
  We need to use frequently the fact that $\Delta$ commutes with translations and
the $K$-averaging operator $\mathcal K$ defined in section 2.
It is clear from the condition $\Delta T_j=
zT_{j+1}$ that if one element of the sequence $\{T_j\}$ is zero, then every elements of the sequence is zero and we have nothing to prove. Therefore we assume that none of the $T_j$ are zero.
We fix $j\in \Z$. We claim that  there  is an $x\in G$ such that
$\ell_x T_j$ has nonzero $K$-invariant  part. Indeed if $\mathcal K(\ell_xT_j)=0$  for all $x\in G$, then
$\langle\ell_x T_j, h_t\rangle=0$ for all $t>0$ since  the heat kernel $h_t$ is a $K$-invariant function (see section 2). That is $T_j\ast
h_t\equiv 0$. But  $T_j\ast h_t\rightarrow T_j$ as $t\rightarrow 0$ in the sense of distribution.
Therefore $T_j=0$ which contradicts our assumption. We note that
this also shows that if for two $L^2$-tempered distribution $T$ and
$T'$, $\mathcal K(\ell_x T)=\mathcal K(\ell_xT')$ for all $x\in G$, then $T=T'$.
\vspace{.1in}

Next we claim that if  $\mathcal K(\ell_y T_0)\neq 0$ for
some $y\in G$, then $\mathcal K(\ell_y T_j)\neq 0$ for all $j\in \Z$. It is
enough to  show that if   $\mathcal K(\ell_y T_0)\neq 0$ then $\mathcal K(\ell_y T_{-1})\neq 0$ and $\mathcal K(\ell_y T_1)\neq
0$. Indeed if $\mathcal K(\ell_y T_{-1})= 0$ then
$\Delta \mathcal K(\ell_y T_{-1})= 0$ which implies $\mathcal K(\ell_y T_0)= 0$ as
$\Delta T_{-1}=zT_0$ for $z\neq 0$. On the other hand, if  $\mathcal K(\ell_y T_1)= 0$, then $\langle
\ell_y T_1, \psi\rangle=0$ for all $\psi\in C^2(G//K)$. That is
$\langle \ell_y \Delta T_0, \psi\rangle=0$
and hence $\langle \ell_y T_0, \Delta
\psi\rangle=0$. Using the characterization of the image of $C^2(G//K)$ under spherical Fourier transform (see section 2) we see that for any $\phi\in C^2(G//K)$,
$\what{\phi}(\lambda)(|\lambda|^2+|\rho|^2)^{-1}\in S(\R^l)_W$. Hence
$\phi$ can be written as $\phi=\Delta\psi$ for some $\psi\in
C^2(G//K)$. Thus  $\langle \ell_y T_0,
\phi\rangle=0$ for any $\phi\in C^2(G//K)$, i.e.  $\mathcal K(\ell_y T_0)=0$. \vspace{.1in}

Our aim now is to  show that for any $y\in G$, the sequence $\{\mathcal K(\ell_y T_j)\}$ of $K$-invariant  distributions satisfies the hypothesis of Theorem \ref{Symm-dist-radial}.
Since $\Delta$ commutes with the $K$-averaging operator  and translations, it follows from the hypothesis
$\Delta T_j=z T_{j+1}$ that $\Delta \mathcal K(\ell_y T_j)=z \mathcal K(\ell_y T_{j+1})$.

It now remains to show that for the seminorm $\nu$ of $C^2(G)$ given in the hypothesis of the theorem  and for any $\psi_1\in C^2(G//K)$, $|\langle \mathcal K(\ell_y T_j), \psi_1\rangle|\le C_y M\nu(\psi_1)$. First we note that for any $y\in G$ and  $\psi\in C^2(G)$, $|\nu(\ell_y\psi)|\le C_y\nu(\psi)$, where the constant $C_y$ depends only on $y$. Indeed, using $\phi_0(x)\le C_y\phi_0(yx)$ for all $x\in G$ (\cite[Proposition 4.6.3., (vi)]{GV}) and triangle inequality $\sigma(yx)\le \sigma(x)+\sigma(y)$, we have,
\begin{eqnarray*}
\nu(\ell_y\psi)&=&\sup_{x\in X}
|(\imath g_1)g_2\,\psi(y^{-1}x)|\phi_0(x)^{-1}(1+\sigma(x))^L\,\, (D\in \mathcal U(\mathfrak g_\C), L>0)\\&=&
\sup_{x\in X} |(\imath g_1)g_2\,\psi(x)|\phi_0(yx)^{-1}(1+\sigma(yx))^L\\
&\le& C \sup_{x\in X} |(\imath g_1)g_2\, \psi(x)|\phi_0(x)^{-1}(1+\sigma(x))^L (1+\sigma(y))^L=C_y\nu(\psi),
\end{eqnarray*} where $L>0$ and $g_1, g_2\in \mathcal U(\mathfrak g_\C)$ are fixed.
Since  $|\langle T_j, \psi\rangle|\le
M\nu(\psi)$ for any $\psi\in C^2(G)$, it follows that  for  $\psi_1\in
C^2(G//K)$,
\begin{eqnarray*}
|\langle \mathcal K(\ell_y T_j), \psi_1\rangle|=|\langle\ell_yT_j,
\psi_1\rangle=|\langle T_j, \ell_{y^{-1}}\psi_1\rangle|\le
M\nu(\ell_{y^{-1}}\psi_1)\le C_y M \nu(\psi_1).
\end{eqnarray*}

From Theorem  \ref{Symm-dist-radial} we conclude that
$$\Delta \mathcal K(\ell_y(T_0))=-|z| \mathcal K(\ell_y(T_0))\text{  for all } y\in G.$$
(Note that if  $\mathcal K(\ell_y(T_0))=0$ for some $y\in G$, then  the identity    $\Delta \mathcal K(\ell_y(T_0))=-|z| \mathcal K(\ell_y(T_0))$ is trivial.)
Again appealing to the fact that
$\Delta$ commutes with translations and $K$-averaging operator  we have
$\mathcal K(\ell_y(\Delta T_0))=\mathcal K(\ell_y(-|z| T_0))$  for all $y\in G$. This implies (see the first paragraph of the proof) that $\Delta T_0=-|z|T_0$ which is the assertion.
\end{proof}

We define $\Delta_1=-(\Delta+|\rho|^2)$. Then a step by step adaptation of the above proof yields the following, which we shall use in the last section.
\begin{theorem} \label{Symm-dist-Delta1}
 If for a doubly infinite sequence
$\{T_j\}$ of $L^2$-tempered distributions on $X$, $\Delta_1 T_j=
zT_{j+1}$ for some nonzero  $z\in \C$  and for a fixed
seminorm $\nu$ of $C^2(X)$, $|\langle T_j, \phi\rangle|\le
M\nu(\phi)$ for some $M>0$ for all $\phi\in C^2(X)$, then $\Delta_1
T_0=|z| T_0$.
\end{theorem}

\subsection{Proof of Theorem \ref{symm-main}} Using  Theorem \ref{Symm-dist} we shall now prove Theorem \ref{symm-main}.
\begin{proof}
By Theorem \ref{Symm-dist}, it suffices to show that  for all $j\in \Z$, $f_j\in C^2(X)'$ and  $|\langle f_j, \phi\rangle|\le C\gamma(\phi)$ for all $\phi\in C^2(X)$, for a fixed seminorm  $\gamma$  of $C^2(X)$ defined by $\gamma(\phi)=\sup_{x\in X} |\phi(x)|(1+\sigma(x))^M \phi_0^{-1}(x)$ with $M>0$  sufficiently large.  Indeed,
\begin{align*} |\int_X f_j(x)\phi(x)dx|&\le C\int_{K\times \mathfrak a^+} |f_j(k \exp H)| |\phi(k\exp H)|J(H)dH dk\\
&=C\gamma(\phi)\int_{\mathfrak a^+}\left(\int_K |f_j(k\exp H)|^p\right)^{1/p}\frac{\phi_0(\exp H)}{(1+|H|)^M} J(H) dH\, dk\\
&\le C\gamma(\phi)\int_{\mathfrak a^+} \phi_0(\exp H)^2 (1+|H|)^{-M} J(H)dH = C\gamma(\phi),
\end{align*}
  Moreover when $\alpha=0$, we apply \cite[Theorem 3.4]{BOS}.
\end{proof}

\section{Roe-Strichartz theorem for Distinguished Laplacian}
The main result of this section is an analogue of Theorem  \ref{symm-main} for a right invariant second order differential operator
which in the context of $\R^l$ is nothing but the Laplace Beltrami
operator of $\R^l$.  This is known as the {\em distinguished
Laplacian} of $X$. We shall make it precise now.  Let $G=NAK$ be the
Iwasawa decomposition of $G$ and $S$ be the solvable Lie group $N\rtimes A$.
We can identify the manifold $S$ with the Riemannian symmetric space
$G/K.$ The image of the $G$-invariant measure on $G/K$ under this
identification corresponds to the left Haar measure on $S$ and the
Riemannian metric on $G/K$ corresponds to a left-invariant metric on
$S.$ In a similar fashion we can identify functions and differential
operators on $G/K$ with those on $S.$ To define the distinguished
Laplacian $\mathcal L$ we first consider the inner product $\langle
X,Y\rangle=B(X,\theta Y)$ on $\frak g$ where $B$ is the Cartan
killing form and $\theta$ is a Cartan involution. With respect to
the above inner product the decomposition $\frak s=\frak
a\oplus\frak n$ is orthogonal. We choose an orthonormal basis
$\{H_1,  \ldots, H_l,  X_1, \ldots ,X_m\}$ of $\frak s$ such that $\text
{span}\{H_1,\ldots ,H_l\}=\frak a$,  $\text
{span}\{X_1,\ldots ,X_m\}=\frak n$ and we view these elements as right
invariant vector fields in the usual way. The distinguished
Laplacian $\mathcal L$ is defined as (see \cite{CGHM, GM})
\[\mathcal L=- [H_1^2+\cdots + H_l^2+\frac 12(X_1^2+\cdots +X_m^2)].\]  The operator $\mathcal L$ is essentially self adjoint on $L^2(S)$ with respect to the left Haar measure of $S$ and enjoys a special relationship with the Laplace-Beltrami operator $\Delta$ when viewed as a left-invariant operator on the solvable group $S$. This relation is explained below as it is crucial for our purpose.
For a function $f$ we define $\wtilde{f}(x)=f(x^{-1})$ for $x\in S$, where $x^{-1}$ is the inversion of the group $S$. We recall that $\Delta_1$ denotes the operator $-(\Delta+|\rho|^2)$.
  It then follows that for all $x\in S$ (see \cite[p.108]{CGHM}),
\begin{equation}\label{relate-Laplacians}
\delta^{1/2}(x)\, (\Delta_1\, \delta^{1/2}\wtilde{f})(x^{-1})\,\,
=\mathcal L f(x), \text{ equivalently } \Delta_1(\delta^{1/2}\wtilde{f})(x)=\delta^{1/2}(x)(\mathcal L f)(x^{-1}),
\end{equation} where we recall $\delta(an)=e^{-2\rho(\log a)}$, for $a\in A$ and $n\in N$. It follows trivially that
$\Delta_1f=\lambda f$ for some $\lambda\in \C$ if and only if  $\mathcal L (\delta^{1/2}\wtilde{f})=\lambda (\delta^{1/2}\wtilde{f})$.
This relation between the Laplacians yields  Theorem \ref{thm-distinguished} stated in the introduction.
\begin{proof}[Proof of Theorem {\em \ref{thm-distinguished}}]
Let $g_j=\delta^{1/2} \wtilde{f}_j$ for all $j\in \Z$.  By
(\ref{relate-Laplacians})
\[\Delta_1 g_j = \Delta_1\delta^{1/2}\wtilde{f}_j=\delta^{1/2}\wtilde{\mathcal L f_j}=
 \alpha\, \delta^{1/2}\, \wtilde{f}_{j+1}=\alpha\, g_{j+1}.\]
It is also clear that $|f_j(x)|<C \delta(x)$ implies $|g_j(x)|\le C
\delta^{-1/2}(x)$.
We recall that $\delta^{-1/2}(x)=e^{-\rho(H(x^{-1}))}$ and  $\mathcal K(\delta^{-1/2})(x)=\phi_0(x)$.  For  a function $\phi\in C^2(X)$,
\begin{align*} |\int_Xg_j(x)\phi(x)dx|&\le C\int_X \delta(x)^{-1/2}|\phi(x)|dx\\&\le C\gamma(\phi)\int_X\delta(x)^{-1/2}\phi_0(x)(1+\sigma(x))^{-M}dx\\ &= C \gamma(\phi)\int_X \phi_0(x)^2 (1+\sigma(x))^{-M}dx =C\gamma(\phi),
\end{align*}
where $\gamma$ is a seminorm of $C^2(X)$ defined by $\gamma(\phi)=\sup_{x\in X} |\phi(x)|(1+\sigma(x))^M \phi_0^{-1}(x)$ for some sufficiently large $M>0$.   Thus the sequence $\{g_j\}$ satisfies the
hypothesis of Theorem \ref{Symm-dist-Delta1} and hence  $\Delta_1 g_0= \alpha\, g_0$. Using
(\ref{relate-Laplacians}) again we get ${\mathcal L} f_0= \alpha f_0$
which is the assertion.
\end{proof}
We conclude with the  observation that despite the fact that the distinguished Laplacian $\mathcal L$  has some similarities with the  usual Laplacian $L$ on $\R^l$ (see Introduction), a straightforward analogue of the Euclidean result of Strichartz \cite{Str} is not a possibility. Following counter example will establish this.

\noindent {\bf Counter Example}: We will produce two bounded
eigenfunctions $\psi_1$ and $\psi_2$ of $\mathcal L$ with
eigenvalues $-4|\rho|^2$ and $4|\rho|^2$ respectively. We can then
define $f_j=(-1)^k\psi_1+\psi_2$, $k\in \Z$. It is then clear that
the above sequence is uniformly bounded with $\mathcal L
(f_j)=4|\rho|^2((-1)^{k+1}\psi_1+\psi_2)=4|\rho|^2 f_{j+1}$ but
$f_0=\psi_1+\psi_2$ is not an eigenfunction of $\mathcal L$. We
define $\psi_1=\delta^{1/2}\phi_{2\rho}.$ Since
$\Delta_1(\phi_{2\rho})=-(\Delta+|\rho|^2I)\phi_{2\rho}=4|\rho|^2\phi_{2\rho}$
it follows from  (\ref{relate-Laplacians}) that $\mathcal L
\psi_1=4|\rho|^2\psi_1$. Since $|\phi_{2\rho}(x)|\leq
C_{\rho}e^{-\rho\sigma(x)}$ and $\sigma(na)\geq |\log a|$ it
follows that
\[|\psi_1(na)|\leq C_{\rho}e^{-\rho \log a}e^{-\rho
|\log a|}\leq C.\] Let $\psi_2$ be the constant function $1$. We
shall show that $\psi_2$ is an eigenfunction of $\mathcal L$ with
eigenvalue $-4|\rho|^2.$ We define
$F_{\lambda}(na)=e^{-(i\lambda+\rho)H(a^{-1}n^{-1})}$ then $\wtilde{
F}_{2i\rho}(na)=e^{\rho H(na)}=e^{\rho\log a}$ and hence
$\delta^{1/2}(na)\wtilde{F}_{2i\rho}(na)=1.$ But since $F_{2i\rho}$ is
an eigenfunction of $\Delta$ with eigenvalue $3|\rho|^2$ it follows
that $\Delta_1 F_{2i\rho}=-4|\rho|^2 F$. Using
(\ref{relate-Laplacians}) we have
\[\mathcal L 1=\mathcal L(\delta^{1/2}\wtilde{F})
=\delta^{1/2}(\Delta_1\delta^{1/2}\delta^{-1/2}F)\wtilde{}=-4|\rho|^2\delta^{1/2}\wtilde{F}=-4|\rho|^2 1.\]
\section{Concluding Remarks}
{\bf 1.} In view of the results in \cite{KRS-2} and in \cite{BOS}, it is natural to  expect the following result.
\begin{conjecture} \label{Conjecture} Fix $q\in (1,2)$.
Let $\{f_j\}_{j\in\N}$ be an  infinite sequence of measurable functions on
$X$ such that  for all $j\in\N$:
 \begin{enumerate}
 \item[(i)] $\Delta f_j=(4\rho^2/qq') f_{j+1}$,
 \item[(ii)] for   a fixed $p\ge 1$, $\|f_j(\cdot a)\|_{L^p(K)}\le C_{p}  \phi_{i\gamma_q\rho}(a)$ for all $a\in \overline{A^+}$ and for a  constant $C_{p}>0$ depending only on  $p$.
 \end{enumerate}
Then $\Delta f_0=-(4\rho^2/qq') f_0$. In particular if  $p>1$ then $f_0(x)=P_{i\gamma_q\rho}F(x)$ for some $F\in L^p(K/M)$ and if $p=1$ then $f_0=P_{i\gamma_q\rho}\mu(x)$ for some signed measure $\mu$ on $K/M$.
\end{conjecture}

{\bf 2.} A recent paper (\cite{NPP}) studies the $L^p$-Schwartz space isomorphisms and related analysis in the context of Heckman-Opdam  hypergeometric functions, which generalizes analysis of $K$-biinvariant functions on a noncompact connected semisimple Lie group with finite centre.   It should be  possible to prove an analogue of our result in this set-up, through similar steps.

\end{document}